\documentclass{article}
\usepackage{amssymb}
\usepackage{amsmath,color}
\usepackage[hidelinks]{hyperref}

\usepackage[square,numbers]{natbib}         
\usepackage{ntheorem}
\newtheorem{lem}{Lemma}[section]
\newtheorem{cor}[lem]{Corollary}
\newtheorem{teo}[lem]{Theorem}

\newtheorem{os}[lem]{Remark}

\newtheorem{prop}[lem]{Proposition}

\newenvironment{proof}{{\sc{Proof.}}}{\hfill\qed}

\newcommand{\qed}{\thinspace\null\nobreak\hfill\hbox{\vbox{\kern-.2pt\hrule
			height.2pt depth.2pt\kern-.2pt\kern-.2pt \hbox to2.5mm{\kern-.2pt\vrule
				width.4pt \kern-.2pt\raise2.5mm\vbox to.2pt{}\lower0pt\vtop
				to.2pt{}\hfil\kern-.2pt \vrule
				width.4pt \kern-.2pt}\kern-.2pt\kern-.2pt\hrule height.2pt depth.2pt
			\kern-.2pt}}\par\medbreak}

\newcommand{\R}{\mathbb{R}}

\newcommand{\N}{\mathbb{N}}

\date{}
\begin{document}

\title{
Gaussian Poincar\'{e} inequalities  on the half-space with singular weights
}
\author{L. Negro \thanks{Dipartimento di Matematica e Fisica ``Ennio De Giorgi'', Universit\`a del Salento, C.P.193, 73100, Lecce, Italy.
e-mail:  luigi.negro@unisalento.it} \qquad C. Spina \thanks{Dipartimento di Matematica e Fisica``Ennio De Giorgi'', Universit\`a del Salento, C.P.193, 73100, Lecce, Italy.
e-mail:  chiara.spina@unisalento.it}}

\maketitle
\begin{abstract}
\noindent 
We prove  Rellich-Kondrachov type theorems and weighted Poincar\'{e} inequalities on  the half-space $\R^{N+1}_+=\{z=(x,y): x \in \R^N, y>0\}$ endowed with the weighted Gaussian measure $\mu :=y^ce^{-a|z|^2}dz$ where   $c+1>0$ and $a>0$. We prove that for some positive constant $C>0$ one has 
\begin{align*}
	\left\|u-\overline u\right\|_{L^2_\mu(\R^{N+1}_+)}\leq  C \|\nabla u\|_{L^2_\mu (\R^{N+1}_+)},\qquad \forall u\in H^1_\mu(\R^{N+1}_+)
\end{align*}
where  $\overline u=\frac 1{\mu(\R^{N+1}_+)}\int_{\R^{N+1}_+} u\,d\mu(z)$.
Besides this we also consider the local case  of bounded domains of $\R^{N+1}_+$ where the measure $\mu$ is $y^cdz$.

\par

\end{abstract}

\section{Introduction}
Poincar\'{e} inequalities are very powerful tools in mathematical analysis which have been extensively
used for the study of PDEs. They
also are  of great  interest in relation with optimal transport, for the study of some aspects of the
topology of some abstract spaces and also in probabilistic problems.  For instance, they can be used for proving existence and  regularity of solutions to certain PDEs as well as Harnack's inequalities,
see the seminal works by De Giorgi \cite{DeGiorgi}, Nash \cite{Nash} and Moser \cite{Moser1, Moser2}.   Counterparts of the Poincar\'{e} inequalities in a weighted setting have gained a lot of attention in the last decades  due to their connection with  the regularity of solutions  to degenerate or singular PDEs  which depend on the particular structure of the weight under consideration (see e.g. \cite{Fabes-Kenig-Serapioni, Franchi-Guti-Wheeden,audrito2024nodalsetsolutionsclass, Heinonen} and references therein).

Following these ideas, in this paper   we extend some recent   results proved in   \cite{audrito2024nodalsetsolutionsclass} by   studying   Rellich-Kondrachov type theorems and weighted Poincar\'{e} inequalities on  the half-space $\R^{N+1}_+=\{z=(x,y): x \in \R^N, y>0\}$ endowed with the weighted Gaussian measure $\mu :=y^ce^{-a|z|^2}dz$ where   $c+1>0$ and $a>0$. Our main results are  Theorems  \ref{Compactness H^1_cg} and  \ref{Poincare weight gauss} where   we prove the compactness of   the immersion $H^1_\mu(\R^{N+1}_+)\hookrightarrow L^2_\mu(\R^{N+1}_+)$ and the validity, for some constant $C>0$, of the  inequality   
\begin{align}\label{eq intr}
	\left\|u-\overline u\right\|_{L^2_\mu(\R^{N+1}_+)}\leq  C \|\nabla u\|_{L^2_\mu (\R^{N+1}_+)},\qquad \forall u\in H^1_\mu(\R^{N+1}_+)
\end{align}
 where    $\overline u=\frac 1{\mu(\R^{N+1}_+)}\int_{\R^{N+1}_+} u\,d\mu(z)$.
Here we set  $L^2_\mu\left(\R^{N+1}_+\right):=L^2(\R^{N+1}_+,d\mu(z))$ and  $$H^{1}_{\mu}(\R^{N+1}_+):=\{u \in L^2_{\mu}(\R^{N+1}_+) : \nabla u \in L^2_{\mu}(\R^{N+1}_+)\}.$$ 
Besides this we also consider the local case  of bounded domains of $\R^{N+1}_+$, see Corollary \ref{Poincare local} and Proposition \ref{Poincare local mean}, where the measure $\mu$ is equivalent to  the weighted measure $\nu=y^cdz$. We  remark that in   \cite[Theorem 4.7, Lemma 5.5]{audrito2024nodalsetsolutionsclass}  the same results are proved in  the smaller  range $-1<c<1$, where the weight $y^c$ is in the Muckenhoupt class $A_2$. The  strategy therein  employed by the authors is based  on a spectral decomposition of the space $L^2_\mu(\R^{N+1}_+)$ which also allows   to obtain  optimal constants. On the other hand our approach   relies  on the derivation of   some suitable weighted Hardy-type inequalities and  has the advantages of being easily applied in the whole range $c>-1$, losing  although  the optimality of the constants.

The motivations for studying the validity of the Poincar\'{e} inequalities \eqref{eq intr} are due to their essential  role in proving, using the method of Nash \cite{Nash,Fabes-Stroock}, Harnack's inequalities and lower heat kernel  estimates for  the   singular   operator 
 \begin{equation*}\label{def L}
	\mathcal L =\sum_{i,j=1}^{N+1}q_{ij}D_{ij}+c\frac{D_y}{y}
\end{equation*}
in  $\R^{N+1}_+$, under Neumann boundary conditions at $y=0$.  Here  $Q\in\R^{N+1,N+1}$ is an elliptic  matrix and $c\in\R$ satisfies $\frac c\gamma +1>0$ where $\gamma=q_{N+1,N+1}$. The interest in this class of singular  operators has grown in the last decade as they appear extensively  in the literature in 
both pure and applied problems  also in connection with  the theory of nonlocal operators (see e.g.  the introductions in \cite{Negro-AlphaDirichlet, dong2020parabolic,Vita2024schauder} and references therein). 
	In the special case $Q=I_{N+1}$ and  $\mathcal L=\Delta_x+D_{yy}+c\frac{D_y}{y}$, these operators  play a major role in the investigation of the fractional powers $(-\Delta_x)^s$ and  $(D_t-\Delta_x)^s$, $s=(1-c)/2$, through the  ``extension procedure" of Caffarelli and Silvestre,  see \cite{Caffarelli-Silvestre} and \cite{Stinga-Torrea-Extension,GaleMiana-Extension,Arendt-ExtensionProblem}. 

We exploit these aspects in \cite{Negro-Spina-SingularKernel-Lower} where, using the Poincar\'{e} inequality \eqref{eq intr}, we prove  that the heat kernel $p_{{\mathcal L}}$ of $\mathcal L$, written  with respect the measure $y^\frac{c}{\gamma}dz$, satisfies the lower Gaussian estimates
\begin{align*}
 p_{{\mathcal  L}}(t,z_1,z_2)
	\geq C t^{-\frac{N+1}{2}} y_1^{-\frac{c}{2\gamma}} \left(1\wedge \frac {y_1}{\sqrt t}\right)^{\frac{c}{2\gamma}} y_2^{-\frac{c}{2\gamma}} \left(1\wedge \frac{y_2}{\sqrt t}\right)^{\frac{c}{2\gamma}}\,\exp\left(-\dfrac{|z_1-z_2|^2}{kt}\right),
\end{align*} 
where $t>0$, $z_1=(x_1,y_1),\ z_2=(x_2,y_2)\in\R^{N+1}_+$ and $C,k$ are some positive constants.  Here it is essential to prove  \eqref{eq intr} in the whole range $c>-1$.
 Similar upper estimates  are proved in \cite{Negro-Spina-SingularKernel} using suitable weighted  Gagliardo-Nirenberg type inequalities. We also remark that elliptic and parabolic solvability of the associated problems in weighted $L^p$ spaces have been  investigated in \cite{MNS-Caffarelli,MNS-CompleteDegenerate,MNS-Singular-Half-Space,MNS-Degenerate-Half-Space, Negro-AlphaDirichlet}   where it is  proved that, under suitable assumptions on $m$ and $p$,  $\mathcal L$ generates an analytic semigroup in $L^p(\R^{N+1}_+; y^m dxdy)$ by describing also  its domain.

\medskip
 We now describe the structure of the paper and the strategy for proving \eqref{eq intr}.

We start by considering, in Section \ref{Poincarè Q }, the local case in $H^1_c(Q)$, where $Q$ is a bounded domain of $R^{N+1}_+$ and  the measure $\mu$ is equivalent to  the weighted measure $\nu=y^cdz$.  Here for simplicity we suppose  $Q=Q_x\times (0,b)$ where $Q_x$ is a bounded open connected subset of $R^N$ and $b>0$ but all the results of this section remain valid also for bounded open connected $Q\subsetneq \R^{N+1}_+$ having Lipschitz boundary. This condition is only required in order to guarantee the  validity of the divergence Theorem and allow integration by parts.   Here we prove   a Rellich-Kondrachov type theorem which assures the compactness of the embedding $H_c^{1}\left(Q\right)\hookrightarrow L^2_c\left(Q\right)$. Here the main difficulty is the singularity or the degeneracy at $y=0$ of the measure $y^cdz$ and to overcome this obstruction we provide some weighted Hardy type inequalities which are of independent interest and which are proved  using integration by parts. We then derive the validity of the Poincar\'{e} inequalities in $H^1_c(Q)$    from the compactness of the embedding  $H_c^{1}\left(Q\right)\hookrightarrow L^2_c\left(Q\right)$ as in the classical theory.   We also prove similar results in the space $H_c^{1,0}\left(Q\right)$ where we impose Dirichlet boundary conditions at $y=b$.

Then we focus,  in Section \ref{Poincarè gaussian }, on the  half-space $\R^{N+1}_+=\{z=(x,y): x \in \R^N, y>0\}$ endowed with the weighted Gaussian measure $\mu :=y^ce^{-a|z|^2}dz$. We adopt a similar strategy but this time we have also to control the behaviour of the density as $|z|\to\infty$. Then using the previous local results and some suitable weighted Hardy type inequalities for large values of $x$ and $y$, we prove  the compactness of the embedding $H^1_\mu(\R^{N+1}_+)\hookrightarrow L^2_\mu(\R^{N+1}_+)$.  The  Poincar\'{e} inequalities in \eqref{eq intr}  follows then similarly as in Section \ref{Poincarè Q }.

\bigskip
\noindent\textbf{Notation.} For $N \ge 0$, $\R^{N+1}_+=\{(x,y): x \in \R^N, y>0\}$. We write $\nabla u$ for the gradient of a function $u$ with respect to all $x,y$ variables and $\nabla_x u, D_y u$ to distinguish the role of $x$ and $y$.
For $c+1>0$ we consider the measure $\nu=y^c dx dy $ in $\R^{N+1}_+$ and  we write $L^2_c(\R_+^{N+1})$, and often only $L^2_c$ when $\R^{N+1}_+$ is understood, for  $L^2(\R_+^{N+1}; y^c dx dy)$. Similarly, for $a>0$, we consider the measure $\mu=y^ce^{-a|z|^2}dz$ and we write $L^2_\mu\left(\R^{N+1}_+\right):=L^2(\R^{N+1}_+, y^ce^{-a|z|^2}dz)$.
Given $a$ and $b$ $\in\R$, $a\wedge b$, $a \vee b$  denote  their minimum and  maximum. We  write $f(x)\simeq g(x)$ for $x$ in a set $I$ and positive $f,g$, if for some $C_1,C_2>0$ 
\begin{equation*}
	C_1\,g(x)\leq f(x)\leq C_2\, g(x),\quad x\in I.
\end{equation*}
If $\mathcal C$ is set of functions defined in $\R^{N+1}_+$ then we denote by $\mathcal C_{\vert Q}$ the set of their restrictions to $Q\subseteq \R^{N+1}_+$. We often write $B^N(0,r)=\{x\in\R^N:|x|\leq r\}$.

\bigskip
\noindent\textbf{Acknowledgment.}
The authors are members of the INDAM (``Istituto Nazionale di Alta Matematica'') research group GNAMPA (``Gruppo Nazionale per l’Analisi Matematica, la Probabilità
e le loro Applicazioni'').

\section{Weighted Poincar\'{e} inequality on bounded domains} \label{Poincarè Q }

We work in the space $L^2_c\left(\R^{N+1}_+\right):=L^2(\R^{N+1}_+, y^c\, dxdy)$ and we assume, throughout the paper, $c+1>0$, so that the measure $d\nu=y^c\, dx\, dy$ is locally finite on $\R^{N+1}_+=\{(x,y): x \in \R^N, y>0\}$. We define the Sobolev space 
$$H^{1}_{c}(\R^{N+1}_+):=\{u \in L^2_{c}(\R^{N+1}_+) : \nabla u \in L^2_{c}(\R^{N+1}_+)\}$$
equipped with the inner product
\begin{align*}
	\left\langle u, v\right\rangle_{H^1_{c}(\R^{N+1}_+)}:= \left\langle u, v\right\rangle_{L^2_{c}(\R^{N+1}_+)}+\left\langle \nabla u, \nabla v\right\rangle_{L^2_c(\R^{N+1}_+)}.
\end{align*} 
Condition $c+1>0$ assures , by \cite[Lemma 11.1]{MNS-Caffarelli}, that any function $u\in  H^{1}_{c}(\R^{N+1}_+)$ has a finite trace $u_0=u(\cdot,0)$ at $y=0$. Moreover, by \cite[Theorem 4.9 and 6.1]{MNS-Sobolev}, the set 
\begin{equation} \label{defC}
	\mathcal{C}:=\left \{u \in C_c^\infty \left(\R^N\times[0, \infty)\right), \ D_y u(x,y)=0\  {\rm for} \ y \leq \delta\ {\rm  and \ some\ } \delta>0\right \},
\end{equation}
is dense in $H^1_c(\R^{N+1}_+)$.
\begin{os}\label{aprrox compact support}
	If a function $u\in H^1_c\left(\R^{N+1}_+\right)$ has  support in $B^N(0,b)\times[0,b]$, then there exists a sequence $\left(u_n\right)_{n\in\N}\in\mathcal C$  such that $ \mbox{supp }u_n\subseteq B^N(0,b)\times[0,b]$  and  $u_n\to u$ in $H^1_c(\R^{N+1}_+)$. For $y$-compact supported functions this follows from  \cite[Remark 4.14]{MNS-Sobolev} but the   same proof of \cite[Theorem 4.9]{MNS-Sobolev} adapted to the simpler case  $H^1_c\left(\R^{N+1}_+\right)$  gives compactness also in the $x$-variable.
\end{os}

\bigskip

We now focus on domains $Q\subsetneq \R^{N+1}_+$.  Here for simplicity we suppose  $Q=Q_x\times (0,b)$ where $Q_x$ is an open subset of $\R^N$ and $b>0$ but all the results of this section are valid for bounded open connected $Q\subsetneq \R^{N+1}_+$ having Lipschitz boundary. This condition is only required in order to guarantee the  validity of the divergence Theorem  and to allow integration by parts. 

As for the global case, we define analogously $L^2_c\left(Q\right):=L^2(Q, y^c\, dxdy)$ and similarly  $	H^{1}_c(Q)=\left\{u\in L^2_c(Q)\ :\  \nabla
u\in L^2_c(Q)\right\}$.  As before, by \cite[Lemma 11.1]{MNS-Caffarelli},  any function $u\in  H^{1}_{c}(Q)$ has  finite traces $u(\cdot,y)$ for any  $0\leq y\leq b$.
We  impose also a Dirichlet boundary condition at $y=b$ by defining
\begin{align*}
	H^{1,0}_c(Q)&=\left\{u\in L^2_c(Q)\ :\  \nabla
	u\in L^2_c(Q)\quad \text{and} \quad u=0 \text{ on } y=b\right\}.
\end{align*}

%

%
%
%

 We start by proving some Hardy-type inequalities for functions in $H^{1,0}_c(Q)$.
 \begin{prop}\label{local Hardy type} 
 	Let $c+1>0$ and let  $0\leq\alpha\leq 1$ such that $\alpha<\frac{c+1}2$. Then  one has 
 	\begin{align*}
 		\left\|\frac u{y^\alpha}\right\|_{L^2_c(Q)}\leq \frac{2 b^{1-\alpha}}{c+1-2\alpha}\left\|D_y u\right\|_{L^2_c(Q)}, \qquad \forall u\in H^{1,0}_c(Q).
 	\end{align*}
 \end{prop}
\begin{proof}
	Let $u\in\ H^{1,0}_c(Q)$. Then integrating by parts we have
	\begin{align*}
		&\int_{Q}|u|^2y^{c-2\alpha}\,dy=\frac{1}{c+1-2\alpha}\int_{Q_x}\int_{0}^b|u|^2D_y\left(y^{c+1-2\alpha}\right)\,dydx\\[1ex]
		&=-\frac{2}{c+1-2\alpha}\int_{Q_x}\int_{0}^buD_y u\,y^{c+1-2\alpha}\,dy+\frac{1}{c+1-2\alpha}\int_{Q_x}\Big[|u(x,y)|^2y^{c+1-2\alpha}\Big]^{y=b}_{y=0}dx.
	\end{align*}
Recalling that, by definition,  $u=0$ on $y=b$ and that  has finite trace at $y=0$ and since $c+1-2\alpha>0$, the last term satisfies
\begin{align*}
\Big[|u(x,y)|^2y^{c+1-2\alpha}\Big]^{y=b}_{y=0}=0.
\end{align*}
Then using H\"older inequality and since $y^{1-\alpha}\leq b^{1-\alpha}$, the previous equality implies
\begin{align*}
	\int_{Q}|u|^2y^{c-2\alpha}\,dy&\leq \frac{2}{c+1-2\alpha}\left\|\frac u {y^\alpha}\right\|_{L^2_c(Q)}\left\|D_y u\,y^{1-\alpha}\right\|_{L^2_c(Q)}\\
	&\leq \frac{2b^{1-\alpha}}{c+1-2\alpha}\left\|\frac u {y^\alpha}\right\|_{L^2_c(Q)}\left\|D_y u\right\|_{L^2_c(Q)}.
\end{align*}
Using, for $\epsilon>0$, the Young inequality 
$$2\left\|\frac u {y^\alpha}\right\|_{L^2_c(Q)}\left\|D_y u\right\|_{L^2_c(Q)}\leq\epsilon \left\|\frac u {y^\alpha}\right\|_{L^2_c(Q)}^2+\frac 1 \epsilon \left\|D_y u\right\|_{L^2_c(Q)}^2$$
 and choosing $\epsilon$ such that $\frac{\epsilon b^{1-\alpha}}{c+1-2\alpha}=\frac 1 2$ we obtain 
 \begin{align*}
 	\left\|\frac u {y^\alpha}\right\|_{L^2_c(Q)}^2\leq \frac{4 b^{2-2\alpha}}{(c+1-2\alpha)^2}\left\|D_y u\right\|_{L^2_c(Q)}^2
 \end{align*}
which is the required claim.
\end{proof}

Choosing  $\alpha=0$ in Proposition \ref{local Hardy type},  we deduce  a Poincar\'{e} type inequality for functions in $H^{1,0}_c(Q)$.

\begin{cor}\label{Poincare local}
	Let $c+1>0$. Then for some positive constant $C>0$  one has
	\begin{align*} 
		\left\|u\right\|_{L^2_c(Q)}\leq \frac{2 b}{c+1}\left\|D_y u\right\|_{L^2_c(Q)},\qquad \forall u\in H^{1,0}_c(Q).
	\end{align*}
\end{cor}

We can now prove a  Rellich-Kondrachov type theorem which assures the compactness of the embedding $H_c^{1,0}\left(Q\right)\hookrightarrow L^2_c\left(Q\right)$.

\begin{prop}\label{Compactness H^1_c}
Let $Q=Q_x\times(0,b)$ where $Q_x$ is a bounded open subset of $\R^N$ and let $c+1>0$. Then the immersion $H_c^{1,0}\left(Q\right)\hookrightarrow L^2_c\left(Q\right)$ is compact.
\end{prop} 
\begin{proof}
	Let $u$ be  in the unit ball 
$\mathcal{B}$ of $H_c^{1,0}\left(Q\right)$ and fix $0<\epsilon<b$. Then, choosing a sufficiently small   $0<\alpha\leq 1$ such that $\alpha<\frac{c+1}2$ and using Proposition \ref{local Hardy type}, one has
\begin{align*}
	\int_{Q_x\cap [0,\epsilon]}|u|^2\, y^c\,dz\leq \epsilon^{2\alpha} \int_{Q_x\cap [0,\epsilon]}\frac{|u|^2}{y^{2\alpha}}\, y^c\,dz \leq  \epsilon^{2\alpha}\frac{4 b^{2-2\alpha}}{(c+1-2\alpha)^2}.
\end{align*}
Since $L^2_c\left(Q_x\times[\epsilon,b]\right)=L^2\left(Q_x\times[\epsilon,b],\,dz\right)$, the compactness of $\mathcal{B}_{\vert Q_x\times[\epsilon,b]}$ in $L^2_c\left(Q_x\times[\epsilon,b]\right)$ is classical. This fact and the above estimate show that $\mathcal{B}$ is totally bounded.
\end{proof}

We  also prove, similarly, the compactness of the embedding $H_c^{1}\left(Q\right)\hookrightarrow L^2_c\left(Q\right)$.

 \begin{prop}\label{local Hardy type ver 2} 
	Let $c+1>0$ and let  $0\leq\alpha\leq 1$ such that $\alpha<\frac{c+1}2$. Then  for some positive constant $C>0$ one has 
	\begin{align*}
		\left\|\frac u{y^\alpha}\right\|_{L^2_c(Q)}\leq C\left\|u\right\|_{H^1_c(Q)}, \qquad \forall u\in H^{1}_c(Q).
	\end{align*}
\end{prop}
\begin{proof}
Let $\eta\in C^\infty_c([0,b[)$ be a cut-off function such that $\eta(y)=1$, if $0\leq y\leq \frac b2$ and let us set
\begin{align*}
	u=\eta u+(1-\eta) u:=u_1+u_2.
\end{align*}
 Since $ u_1\in H^{1,0}_c\left(Q\right)$, we can apply Proposition \ref{local Hardy type} thus obtaining for some positive constant $C>0$ (which may be different in each occurrence)
\begin{align*}
	\left\|\frac {u_1}{y^\alpha}\right\|_{L^2_c(Q)}&\leq C\left\|D_y u_1 \right\|_{L^2_c( Q)}\leq C\|u\|_{H^{1}_c\left(Q\right)}.
\end{align*}
On the other hand  since $ \mbox{supp}(u_2)\subseteq \overline{Q_x}\times [\frac b 2,b ]$, we have trivially 
$$\left\|\frac {u_2}{y^\alpha}\right\|_{L^2_c(Q)}=\left\|\frac {u_2}{y^\alpha}\right\|_{L^2_c( Q_x\times (\frac b 2,b))}\leq\left(\frac 2{b}\right)^{\alpha} \left\|u\right\|_{L^2_c(Q)}$$
which added to the previous inequality implies
\begin{align*}
	\left\|\frac {u}{y^\alpha}\right\|_{L^2_c(Q)}
	&\leq C\left\| u\right\|_{H^1_c(Q)}
\end{align*}
	which is the required claim.
\end{proof}

\begin{prop}\label{Compactness H^1_c ver mean}
	Let $Q=Q_x\times(0,b)$ where $Q_x$ is a bounded open subset of $\R^N$ and let $c+1>0$. Then the immersion $H_c^{1}\left(Q\right)\hookrightarrow L^2_c\left(Q\right)$ is compact.
\end{prop} 
\begin{proof} We proceed as in the proof of Proposition \ref{Compactness H^1_c}.
	Let $u$ be  in the unit ball 
	$\mathcal{B}$ of $H_c^{1}\left(Q\right)$ and fix $0<\epsilon<b$. Then, choosing a sufficiently small   $0<\alpha\leq 1$ such that $\alpha<\frac{c+1}2$ and using Proposition \ref{local Hardy type ver 2}, one has
	\begin{align*}
		\int_{Q_x\cap [0,\epsilon]}|u|^2\, y^c\,dz\leq \epsilon^{2\alpha} \int_{Q_x\cap [0,\epsilon]}\frac{|u|^2}{y^{2\alpha}}\, y^c\,dz \leq C \epsilon^{2\alpha}.
	\end{align*}
	Since $L^2_c\left(Q_x\times[\epsilon,b]\right)=L^2\left(Q_x\times[\epsilon,b],\,dz\right)$ and $H^1_c\left(Q_x\times[\epsilon,b]\right)=H^1\left(Q_x\times[\epsilon,b],\,dz\right)$, the compactness of $\mathcal{B}_{\vert Q_x\times[\epsilon,b]}$ in $L^2_c\left(Q_x\times[\epsilon,b]\right)$ is classical. This fact  and the above estimate show that $\mathcal{B}$ is totally bounded.
\end{proof}

We can then  deduce a Poincar\'{e} type inequality for functions in $H^{1}_c(Q)$.

\begin{prop}\label{Poincare local mean}
	Let $c+1>0$ and let $Q=Q_x\times(0,b)$ where $Q_x$ is a bounded open connected subset of $\R^N$. Then, setting $\nu=y^cdz$, one has for some positive constant $C>0$
	\begin{align*}
		\left\|u-\overline u\right\|_{L^2_c(Q)}\leq  C \|\nabla u\|_{L^2_c(Q)},\qquad \forall u\in H^{1}_c(Q),\quad \text{where}\quad \overline u=\frac 1{\nu(Q)}\int_Q u\,y^cdz.
	\end{align*}
\end{prop}
\begin{proof}
	The proof is an adaptation of the standard proof of the  Poincar\'{e}-Wirtinger inequality (see e.g. \cite[Section 5.8, Theorem 1, page 275]{evans-Book-PDE}).   We argue by contradiction and we assume that there exists a sequence $\left(u_n\right)_{n\in\N}\subseteq H^{1}_c(Q)$ such that 
	\begin{align*}
		\left\|u_n\right\|_{L^2_c(Q)}>  n \|\nabla u_n\|_{L^2_c(Q)}.
	\end{align*} 
Up to replace $u_n$ with $u_n-\overline {u}_n$, we can suppose $\overline u_n=0$.   Furthermore, up to renormalize, we can  assume, without any loss of generality,  $\left\|u_n\right\|_{L^2_c(Q)}=1$, which implies 
\begin{align}\label{Poincare mean1}
	 \|\nabla u_n\|_{L^2_c (Q)}<\frac 1 n.
\end{align} By Proposition \ref{Compactness H^1_c ver mean}, there exists a subsequence $(u_{n_k})\subseteq (u_n)$ which converges strongly to a function $u\in L^2_c (Q)$  and weakly in $H^1_c (Q)$. The limit $u$ then  satisfies clearly  $\|u\|_{L^2_c(Q)}=1$ and, using \eqref{Poincare mean1}, $u\in H^{1}_c(Q)$ and moreover  $\nabla u=0$.   This  implies  then $u$ to be constant and, since $\overline u=0$, we must have $u=0$ which contradicts $\|u\|_{L^2_c(Q)}=1$. 
\end{proof}

\section{Weighted Gaussian Poincar\'{e} inequality on $\R^{N+1}_+$} \label{Poincarè gaussian }
In this section we prove   Rellich-Kondrachov type theorems  and  Poincar\'{e} type inequalities in weighted Gaussian spaces. We fix $c+1>0$ and $a>0$ and we consider the weighted  Gaussian measure $\mu :=y^ce^{-a|z|^2}dz$ on $\R^{N+1}_+$ which, from the assumptions on $a$ and $c$, is  finite on $\R^{N+1}_+$. By standard computation one has
\begin{align*}
	\mu\left(\R^{N+1}_+\right)=\int_{\R^{N+1}_+} y^c e^{-a|z|^2}\,dz= a^{-\frac{N+1+c}{2}}\pi^{\frac N2}\Gamma\left(\frac{c+1}2\right)
\end{align*} 
where $\Gamma$ is the Gamma function. Indeed
\begin{align*}
	\int_{\R^{N+1}_+} y^c e^{-a|z|^2}\,dz&=\int_{\R^{N}} e^{-a|x|^2}\,dx\,\int_0^\infty y^c e^{-ay^2}\,dy\\[1ex]
	&\overset{s=ay^2}{=} \left(\frac{\pi}{a}\right)^{\frac N2}\int_0^\infty s^{\frac {c-1}2} e^{-s}\,dy\,a^{-\frac{c+1}2}=a^{-\frac{N+1+c}{2}}\pi^{\frac N2}\Gamma\left(\frac{c+1}2\right).
\end{align*} We work in the space $L^2_\mu\left(\R^{N+1}_+\right):=L^2(\R^{N+1}_+, y^ce^{-a|z|^2}dz)$ and we define the Sobolev space 
$$H^{1}_{\mu}(\R^{N+1}_+):=\{u \in L^2_{\mu}(\R^{N+1}_+) : \nabla u \in L^2_{\mu}(\R^{N+1}_+)\}$$
equipped with the inner product
\begin{align*}
	\left\langle u, v\right\rangle_{H^1_{\mu}(\R^{N+1}_+)}:= \left\langle u, v\right\rangle_{L^2_{\mu}(\R^{N+1}_+)}+\left\langle \nabla u, \nabla v\right\rangle_{L^2_\mu(\R^{N+1}_+)}.
\end{align*} 
As in the previous section, condition $c+1>0$ assures, by \cite[Lemma 11.1]{MNS-Caffarelli}, that any function $u\in  H^{1}_{\mu}(\R^{N+1}_+)$ has a finite trace $u_0=u(\cdot,0)$ at $y=0$. 

The next results  show the density  of smooth functions in $H^1_\mu(\R^{N+1}_+)$.

\begin{prop}\label{core Gaussian}
	 The set  
	\begin{equation*} 
		\mathcal{C}:=\left \{u \in C_c^\infty \left(\R^N\times[0, \infty)\right), \ D_y u(x,y)=0\  {\rm for} \ y \leq \delta\ {\rm  and \ some\ } \delta>0\right \}
	\end{equation*}
	is dense in $H^1_\mu(\R^{N+1}_+)$.
\end{prop}
\begin{proof}
Let us prove preliminarily that functions in  $H^1_\mu(\R^{N+1}_+)$ having  compact support in $\overline{\R^{N+1}_+}$  are dense in  $H^1_\mu(\R^{N+1}_+)$. Let  $u\in H^1_\mu(\R^{N+1}_+)$.
 Let $0\leq\phi\leq 1$ be a smooth function depending only on the $y$ variable which is equal to $1$ in $(0,1)$ and to $0$ for $y \ge 2$ and set $\phi_n(y)=\phi \left(\frac{y}{n}\right)$. Similarly let  $0\leq\eta\leq 1$ be a smooth function depending only on the $x$ variable which is equal to $1$ if $|x|\leq 1$ and to $0$ for $|x| \ge 2$ and set $\eta_n(x)=\phi \left(\frac{x}{n}\right)$.
 Let  $u_n(x,y)=\eta_n(x)\phi_n(y)u(x,y)$. Then $u_n\in H^1_\mu(\R^{N+1}_+)$ and has compact support  in $B^N(0,2n) \times [0,2n]$ where we set $B^N(0,r)=\{x\in\R^N:|x|\leq r\}$. By the dominated convergence theorem we have $u_n \to u$ in $L^2_{\mu}(\R^{N+1}_+)$. Concerning the derivatives we observe that 
\begin{align*}
	D_y u_n =\eta_n\phi_n D_y u+ D_y \phi_n \eta_n u,\qquad |D_y \phi_n|\leq \frac{C}{n}\chi_{[n,2n]}\\
	\nabla_xu_n=\eta_n\phi_n \nabla_xu+ \nabla_x\eta_n\phi_n u,\qquad |\nabla_x \eta_n(x)|\leq \frac{C}{n}\chi_{[n,2n]}(|x|)
\end{align*}
which implies  $D_y u_n \to D_y u$, $ \nabla_{x}u_n\to \nabla_{x}u$  in $L^2_{\mu}(\R^{N+1}_+)$ by the dominated convergence theorem again.

Let us now prove  that   $\mathcal C$  is dense in  $H^1_\mu(\R^{N+1}_+)$.
Let  $u\in H^1_\mu(\R^{N+1}_+)$. By the previous step we may assume that $u$ has support in  $B^N(0,b) \times [0,b[$ for some $b>0$. By  \cite[Theorem 6.1]{MNS-Sobolev} and  Remark \ref{aprrox compact support} there exists $u_n\in\mathcal C$ such that $u_n\to u$ in $H^1_c(B^N(0,b) \times [0,b[)$. Since the measure $\mu$ is equivalent to the measure $y^c dz$ in  $B^N(0,b) \times [0,b[$ we then have $H^1_c(B^N(0,b) \times [0,b[)=H^1_\mu (B^N(0,b) \times [0,b[)$. This proves the required claim.
\end{proof}

\begin{os}
	As in Remark \ref{aprrox compact support}, the above proof shows that if  $u\in H^1_\mu\left(\R^{N+1}_+\right)$ has  support in $B^N(0,b)\times[0,b]$, then there exists a sequence $\left(u_n\right)_{n\in\N}\in\mathcal C$  such that $ \mbox{supp }u_n\subseteq B^N(0,b)\times[0,b]$  and  $u_n\to u$ in $H^1_\mu(\R^{N+1}_+)$. Moreover, in the range $c\geq 1$, by using  \cite[Proposition 11.4]{MNS-Caffarelli} in place of \cite[Theorem 6.1]{MNS-Sobolev}, the same proof of Proposition \ref{core Gaussian} proves that $C_c^\infty\left(\R^{N+1}_+\right)$ is dense in $H^1_\mu(\R^{N+1}_+)$
\end{os}

We state now some weighted Hardy-type inequalities for functions in $H^1_\mu(\R^{N+1}_+)$. The following Proposition is an adaptation on $\R^{N+1}_+$ of a classical result for Gaussian measure, see e.g.\cite[Lemma 8.5.2]{Bertoldi_Lorenzi}.

\begin{prop}\label{Hardy type gaussian x}
For some positive constant $C>0$   one has
\begin{align*}
	\left\|u|x| \right\|_{L^2_\mu(\R^{N+1}_+)}&\leq C\left( \left\|\nabla_x u\right\|_{L^2_\mu(\R^{N+1}_+)}+\left\| u\right\|_{L^2_\mu(\R^{N+1}_+)}\right), \qquad \forall u\in H^1_\mu(\R^{N+1}_+).
\end{align*}
\end{prop}
\begin{proof}
By Proposition \ref{core Gaussian}, we can assume, without any loss of generality, $u\in\mathcal C$.	Integrating by parts one has 
	\begin{align*}
		\int_{\R^{N+1}_+}|u|^2|x|^2\,d\mu(z)&=\int_{\R^{N+1}_+}|u|^2|x|^2y^{c}e^{-a|z|^2}\,dydx\\[1ex]
		&=-\frac{1}{2a}\int_{\R^{N+1}_+}|u|^2\left(x,\nabla_x(e^{-a|z|^2}) \right)\,y^cdydx=\frac{1}{2a}\int_{\R^{N+1}_+}\mbox{div}_x\left(|u|^2  x\right)\,d\mu(z)\\[1ex]
		&=\frac{1}{a}\int_{\R^{N+1}_+}u\left(\nabla_x u,x\right)\,d\mu(z)+\frac{N}{2a}\int_{\R^{N+1}_+} |u|^2\,d\mu(z).
	\end{align*}
	Then using the H\"older and the  Young inequalities, the previous equality implies
	\begin{align*}
		\int_{\R^{N+1}_+}|u|^2|x|^2\,d\mu(z)&\leq \frac{1}{a}\left\|u|x|\right\|_{L^2_\mu(\R^{N+1}_+)}\left\|\nabla_x u\right\|_{L^2_\mu(\R^{N+1}_+)}+\frac{N}{2a}\int_{\R^{N+1}_+} |u|^2\,d\mu(z)\\[1ex]
		&\leq\frac{\epsilon}{a} \left\|u|x|\right\|_{L^2_\mu(\R^{N+1}_+)}^2+\frac 1 {a\epsilon} \left\|\nabla_x u\right\|_{L^2_\mu(\R^{N+1}_+)}^2+\frac{N}{2a}\left\| u\right\|_{L^2_\mu(\R^{N+1}_+)}^2.
	\end{align*}
	choosing $\frac \epsilon a =\frac 1 2 $ we obtain the required claim.
\end{proof}

In the following Lemma we restrict ourselves to functions vanishing at the boundary $y=0$. 

\begin{lem}\label{Lemma Hardy y grandi}
For some positive constant $C>0$   one has
\begin{align*}
	\left\|uy \right\|_{L^2_\mu(\R^{N+1}_+)}&\leq C\left( \left\|D_y u\right\|_{L^2_\mu(\R^{N+1}_+)}+\left\| u\right\|_{L^2_\mu(\R^{N+1}_+)}\right),\qquad \forall  u\in C^\infty_c(\R^{N+1}_+).
\end{align*}
\end{lem}
\begin{proof}
We proceed  as in the proof of Proposition \ref{Hardy type gaussian x}. We integrate by parts thus obtaining
 \begin{align*}
 	\int_{\R^{N+1}_+}|u|^2y^2\,d\mu(z)&=\int_{\R^{N+1}_+}|u|^2y^{c+2}e^{-a|z|^2}\,dydx 	=-\frac{1}{2a}\int_{\R^{N+1}_+}|u|^2y^{c+1}D_y(e^{-a|z|^2})\,dydx\\[1ex]
 	&=\frac{1}{a}\int_{\R^{N+1}_+}uD_y u\,y\,d\mu(z)+\frac{c+1}{2a}\int_{\R^{N+1}_+} |u|^2\,d\mu(z).
 \end{align*}
 The conclusion then follows using the H\"older and the  Young inequalities as in the previous proposition.
\end{proof}

We now consider, for $\alpha\in\R$, the weight
\begin{align}\label{weight in Hardy}
	y\left(1\wedge y\right)^{-\alpha-1}=\begin{cases}
		y,\quad &y\geq 1,\\[1ex]
		y^{-\alpha},\quad &y<1,
	\end{cases},\qquad \forall y>0.
\end{align}

\begin{prop}\label{Hardy type gaussian y}
	Let $c+1>0$ and let  $0\leq\alpha\leq 1$ such that $\alpha<\frac{c+1}2$. Then for some positive constant $C>0$   one has
	\begin{align*}
	\left\|u y\left(1\wedge y\right)^{-\alpha-1} \right\|_{L^2_\mu(\R^{N+1}_+)}\leq C\left\|u\right\|_{H^1_\mu(\R^{N+1}_+)}, \qquad \forall u\in H^1_\mu(\R^{N+1}_+).
\end{align*}
\end{prop}
\begin{proof}
	Let $u\in H^1_\mu(\R^{N+1}_+)$. By Proposition \ref{core Gaussian} and using a density argument, we can suppose without any loss of generality $u\in\mathcal C$.  Let $\eta\in C^\infty_c([0,1[)$ be a cut-off function such that $\eta(y)=1$, if $0\leq y\leq \frac 12$ and let us set
	\begin{align*}
		u=\eta u+(1-\eta) u:=u_1+u_2.
	\end{align*}
We focus, preliminarily,  on $0\leq y\leq 1$. Since $ u_1\in C^\infty_c\left( \R^N\times [0,1[\right)$, we can apply Proposition \ref{local Hardy type} thus obtaining	
	\begin{align*}
		\left\|\frac {u_1}{y^\alpha}\right\|_{L^2_c( \R^N\times [0,1[)}&\leq \frac{2 }{c+1-2\alpha}\left\|D_y u_1 \right\|_{L^2_c( \R^N\times [0,1[)}
	\end{align*}
	which, observing that $e^{-a}\leq e^{-ay^2}\leq 1$ for $y\in[0,1]$, implies
	\begin{align*}
		\left\|\frac {u_1}{y^\alpha}\right\|_{L^2_\mu( \R^N\times [0,1[)}&\leq \frac{2e^{a}}{c+1-2\alpha}\left\|D_y u_1\right\|_{L^2_\mu( \R^N\times [0,1[)}\\[1ex]
		&\leq \frac{2 e^{a}}{c+1-2\alpha}\left(\left\|D_y u\right\|_{L^2_\mu(\R^{N+1}_+)}+\left\| u\right\|_{L^2_\mu(\R^{N+1}_+)}\right).
	\end{align*}
On the other hand  since $ u_2\in C^\infty_c\left( \R^N\times [\frac 1 2,+\infty [\right)$, we have trivially 
$$\left\|\frac {u_2}{y^\alpha}\right\|_{L^2_\mu( \R^N\times [0,1[)}=\left\|\frac {u_2}{y^\alpha}\right\|_{L^2_\mu( \R^N\times [\frac 1 2,1[)}\leq{2^{\alpha}} \left\|u\right\|_{L^2_\mu(\R^{N+1}_+)}$$
which added to the previous inequality implies
\begin{align}\label{weigted Hardy Gauss eq 1}
	\left\|\frac {u}{y^\alpha}\right\|_{L^2_\mu( \R^N\times [0,1[)}
	&\leq C\left(\left\|D_y u\right\|_{L^2_\mu(\R^{N+1}_+)}+\left\| u\right\|_{L^2_\mu(\R^{N+1}_+)}\right).
\end{align}
We now focus  on $y\geq 1$ and we argue similarly. We apply  Lemma \ref{Lemma Hardy y grandi} to $ u_2$ thus obtaining	
	\begin{align*}
		\left\|u_2 y\right\|_{L^2_\mu( \R^N\times [\frac 1 2,+\infty [)}&\leq C\left( \left\|D_y u_2\right\|_{L^2_\mu(\R^{N+1}_+)}+\left\| u_2\right\|_{L^2_\mu(\R^{N+1}_+)}\right)\\[1ex]
		&\leq C\left( \left\|D_y u\right\|_{L^2_\mu(\R^{N+1}_+)}+\left\| u\right\|_{L^2_\mu(\R^{N+1}_+)}\right)
	\end{align*}
Since  we have trivially 
$$\left\|u_1y\right\|_{L^2_\mu( \R^N\times [\frac 1 2,+\infty[)}=\left\| {u_1}y\right\|_{L^2_\mu( \R^N\times [\frac 1 2,1[)}\leq \left\|u\right\|_{L^2_\mu(\R^{N+1}_+)}$$
summing up the previous inequalities yields
\begin{align}\label{weigted Hardy Gauss eq 2}
	\left\|u y\right\|_{L^2_\mu( \R^N\times [\frac 1 2,+\infty[)}
	&\leq C\left(\left\|D_y u\right\|_{L^2_\mu(\R^{N+1}_+)}+\left\| u\right\|_{L^2_\mu(\R^{N+1}_+)}\right).
\end{align}
Inequities \eqref{weigted Hardy Gauss eq 1} and \eqref{weigted Hardy Gauss eq 2} prove the required claim.
\end{proof}

We can now prove a  Rellich-Kondrachov type theorem which assures the compactness of the embedding $H^1_\mu(\R^{N+1}_+)\hookrightarrow L^2_\mu(\R^{N+1}_+)$.

\begin{teo}\label{Compactness H^1_cg}
	Let $c+1>0$. The immersion $H^1_\mu(\R^{N+1}_+)\hookrightarrow L^2_\mu(\R^{N+1}_+)$ is compact.
\end{teo} 
\begin{proof}
	We have to prove that the unit ball 	$\mathcal{B}$ of $H^1_\mu(\R^{N+1}_+)$ is totally bounded in $L^2_\mu$. Let us fix $\epsilon>0$ and let  $r>1$ to be chosen later. Let us  also write $B_r=\{x\in\R^N:|x|\leq r\}$ and $B_r^c=\{x\in\R^N:|x|> r\}$. 
	
	Let us consider, preliminarily, $\mathcal{B}_{\vert B_r\times [\frac 1 r,r]}$. Since in $B_r\times [\frac 1 r,r]$ the measure $\mu$ is equivalent to the Lebesgue measure then  $L^2_\mu\left(B_r\times [\frac 1 r,r]\right)=L^2\left(B_r\times [\frac 1 r,r],\,dz\right)$ and $H^1_\mu\left(B_r\times [\frac 1 r,r]\right)=H^1\left(B_r\times [\frac 1 r,r],\,dz\right)$: the compactness of $\mathcal{B}_{B_r\times [\frac 1 r,r]}$ in $L^2_\mu\left(B_r\times [\frac 1 r, r]\right)$  is therefore classical and, consequently, $\mathcal{B}_{\vert B_r\times [\frac 1, r,r]}$ is totally bounded. Thus there exists a finite number of functions $f_1,\dots, f_n\in L^2_\mu\left(B_r\times [\frac 1 r,r]\right)$ such that
	\begin{align*}
	\mathcal{B}_{\vert B_r\times [\frac 1 r,r]}\subseteq \bigcup_{i=1}^n\left\{f\in L^2_\mu\left(B_r\times [\frac 1 r,r]\right):\|f-f_i\|_{L^2_\mu\left(B_r\times [\frac 1 r,r]\right)}<\epsilon\right\}
	\end{align*}
	
Let us treat now $\mathcal{B}_{\vert \left(B_r\times [\frac 1 r,r]\right)^c}$ and let us fix, with this aim,  a sufficiently small $0<\alpha\leq 1$ such that $\alpha<\frac{c+1}2$.	By Propositions \ref{Hardy type gaussian x} and \ref{Hardy type gaussian y} we have
	\begin{align}\label{Compactness H^1_cg eq 1}
		\left\|u|x| \right\|_{L^2_\mu(\R^{N+1}_+)}+	\left\|u y\left(1\wedge y\right)^{-\alpha-1} \right\|_{L^2_\mu(\R^{N+1}_+)}&\leq C\left\| u\right\|_{H^1_\mu(\R^{N+1}_+)}, \qquad \forall u\in H^1_\mu(\R^{N+1}_+).
	\end{align}
Let $u\in\mathcal{B}$. We  distinguish the following cases. 
	\begin{itemize}
	\item[(i)] Let us consider $\mathcal{B}_{\vert \R^N\times [0,\frac 1 r]}$.  Then by \eqref{Compactness H^1_cg eq 1} we have  	
	\begin{align*}
		\|u\|_{L^2_\mu\left(\R^N\times [0,\frac 1 r]\right)}
\leq \frac{1}{r^{\alpha}}\|uy^{-\alpha}\|_{L^2_\mu\left(\R^N\times [0,\frac 1r]\right)}=\frac{1}{r^\alpha}\|uy\left(1\wedge y\right)^{-\alpha-1}\|_{L^2_\mu\left(\R^N\times [0,\frac 1 r]\right)}\leq C\frac 1{r^\alpha}.
	\end{align*}
		\item[(ii)] Let us consider $\mathcal{B}_{\vert \R^N\times [r,+\infty[}$.  Then similarly by \eqref{Compactness H^1_cg eq 1} we have  	
		\begin{align*}
			\|u\|_{L^2_\mu\left(\R^N\times [r,+\infty[\right)}
			\leq \frac 1 r\|uy\|_{L^2_\mu\left(\R^N\times [r,+\infty[\right)}= \frac 1r\|uy\left(1\wedge y\right)^{-\alpha-1}\|_{L^2_\mu\left(\R^N\times [r,+\infty[\right)}\leq C\frac 1r.
		\end{align*}
\item[(iii)] Let us consider $\mathcal{B}_{\vert B_r^c\times [\frac 1 r,r]}$. By \eqref{Compactness H^1_cg eq 1} we have  	
\begin{align*}
	\|u\|_{L^2_\mu\left(B_r^c\times [\frac 1r ,r]\right)}
	\leq \frac 1 r \|u|x|\|_{L^2_\mu\left(B_R^c\times [\frac 1r,r]\right)}\leq C\frac 1 r.
\end{align*}
\end{itemize}
The above estimates show that, taking a sufficiently large $r$, we get
\begin{align*}
		\|u\|_{L^2_\mu\left(\left(B_r\times [\frac 1 r,r]\right)^c\right)}
	\leq \epsilon 
\end{align*}
Therefore extending the functions $f_i$ by zero to the whole $\R^{N+1}_+$ we obtain
\begin{align*}
	\mathcal{B}\subseteq \bigcup_{i=1}^n\left\{f\in L^2_\mu\left(\R^{N+1}_+\right):\|f-f_i\|_{L^2_\mu\left(\R^{N+1}_+\right)}<2\epsilon\right\}
\end{align*}
which implies that $\mathcal{B}$ is totally bounded in $L^2_\mu\left(\R^{N+1}_+\right)$.
\end{proof}

As in the classical theory, we can easily deduce, from the compactness of the embedding  $H^1_\mu(\R^{N+1}_+)\hookrightarrow L^2_\mu(\R^{N+1}_+)$,  a Poincar\'{e}-type inequality for functions in $H^1_\mu(\R^{N+1}_+)$.

\begin{teo}\label{Poincare weight gauss}
	Let $c+1>0$. Then one has for some positive constant $C>0$
	\begin{align*}
		\left\|u-\overline u\right\|_{L^2_\mu(\R^{N+1}_+)}\leq  C \|\nabla u\|_{L^2_\mu (\R^{N+1}_+)},\qquad \forall u\in H^1_\mu(\R^{N+1}_+),
	\end{align*}
where $\overline u=\frac 1{\mu(\R^{N+1}_+)}\int_{\R^{N+1}_+} u\,d\mu(z)$.
\end{teo}
	\begin{proof}
	The proof is analogous to the one of Proposition \ref{Poincare local mean}. We sketch only the main steps. Up to replace $u$ with $u-\overline u$, we can suppose $\overline u=0$.   Then arguing by contradiction, we can assume that there exists a sequence $\left(u_n\right)_{n\in\N}\subseteq H^1_\mu (\R^{N+1}_+)$ satisfying  
	\begin{align}\label{Poincare g1}
		\left\|u_n\right\|_{L^2_\mu (\R^{N+1}_+)}=1,\qquad \|\nabla u_n\|_{L^2_\mu (\R^{N+1}_+)}<\frac 1 n.
	\end{align} By Theorem \ref{Compactness H^1_cg}, there exists a subsequence $(u_{n_k})\subseteq (u_n)$ which converges strongly to a function $u\in L^2_\mu (\R^{N+1}_+)$  and weakly in $H^1_\mu (\R^{N+1}_+)$. The limit $u$ satisfies, clearly,  $\|u\|_{L^2_\mu (\R^{N+1}_+)}=1$ and, using \eqref{Poincare g1}, we can also deduce that $u\in H^1_\mu (\R^{N+1}_+)$ and $\nabla u=0$ which implies $u$ to be constant. Since $\overline u=0$, we must have $u=0$ which contradicts $\|u\|_{L^2_\mu (\R^{N+1}_+)}=1$. 
	
\end{proof}

\bibliography{../TexBibliografiaUnica/References}
\end{document}